\documentclass{article}

\usepackage{amsmath,amsfonts,latexsym,amssymb,amsthm,url, titling}

\newtheorem{theorem}{Theorem}[section]
\newtheorem{proposition}[theorem]{Proposition}
\newtheorem{corollary}[theorem]{Corollary}
\newtheorem{lemma}[theorem]{Lemma}
\newtheorem{remark}[theorem]{Remark}

\DeclareMathOperator{\height}{h}
\DeclareMathOperator{\heightp}{h_p}

\newcommand{\C}{\mathbb{C}}
\newcommand{\Q}{\mathbb{Q}}
\newcommand{\R}{\mathbb{R}}
\newcommand{\Z}{\mathbb{Z}}

\newcommand{\NN}{\mathcal{N}}

\newcommand{\eps}{\varepsilon}

\newcommand{\Mu}{\textrm{M}}

\newcommand{\tho}{\text{th}}

\title{Values of generalized Liouville power series at algebraic numbers}

\thanksmarkseries{alph}
\author{Yu.~Bilu\thanks{Supported by the ANR project JINVARIANT.}, D.~Marques\thanks{Supported by the Franco-Brazilian Mathematics Network.}, C.~G.~Moreira\thanks{Partially supported by FAPERJ and CNPq.}}

\date{\today}

\numberwithin{equation}{section}

\makeatletter

\renewcommand*\l@section[2]{
	\ifnum \c@tocdepth >\z@
	\addpenalty\@secpenalty
	\addvspace{0.2em \@plus\p@}
	\setlength\@tempdima{1.5em}
	\begingroup
	\parindent \z@ \rightskip \@pnumwidth
	\parfillskip -\@pnumwidth
	\leavevmode \bfseries
	\advance\leftskip\@tempdima
	\hskip -\leftskip
	#1\nobreak\hfil \nobreak\hb@xt@\@pnumwidth{\hss #2}\par
	\endgroup
	\fi}

\makeatother

\begin{document}
\hfuzz=4pt

\maketitle

\begin{abstract}

For every positive integer~$m$ LeVeque (1953)  defined the $U_m$-numbers as the transcendental numbers that admit very good approximation by algebraic numbers of degree~$m$, but not by those of smaller degree. In these terms,  Mahler's $U$-numbers are the transcendental numbers which are $U_m$ for some~$m$. In 1965 Mahler showed  that (properly defined) lacunary power series  with integers coefficients  take $U$-values at algebraic numbers, unless the value is algebraic for an obvious reason. However, his argument does not specify to which $U_m$ the value belongs.

In this article, we introduce the notion of \textit{generalized Liouville series}, and give a necessary and sufficient condition for their values to be $U_m$. As an application, we show that a generalized Liouville series takes a $U_m$-value at a simple algebraic integer of degree~$m$, unless the value is algebraic for an obvious reason.  (An algebraic number~$\alpha$ is called \textit{simple} if the number field $\Q(\alpha)$ does not have a proper subfield other than~$\Q$.) 
\end{abstract}

 {\footnotesize
 
\tableofcontents

}

\section{Introduction}
\label{sintro}

In 1844 Liouville~\cite{Li44,Li51} showed that  ${\sum_{n=1}^\infty a^{-n!}}$ is a transcendental number for any integer ${a>1}$; this was the historically first proof of existence of transcendental numbers.  He proved it by showing that an  algebraic number cannot be too well approximated by rationals, while the sum above can be very well approximated by  the finite sums ${a^{-1!}+\cdots+ a^{-n!}}$.   Real numbers that admit too good approximation by rationals are called \textit{Liouville numbers}. Here is the precise definition: ${\alpha \in \R}$ is a Liouville number if for every ${w>0}$ there exist infinitely many ${p/q\in \Q}$ such that ${|\alpha-p/q|<q^{-w}}$.

Mahler~\cite{Ma32} and LeVeque~\cite{Le53} generalized the notion of Liouville numbers, by introducing  $U$-numbers and $U_m$-numbers for every positive integer~$m$. To recall their definitions, let us fix some notation  and terminology. Given a polynomial  ${P(T):=a_mT^m+\cdots+a_0\in \C[T]}$, we denote by $|P|$   the maximal absolute value of its coefficients:
${|P|:=\max\{|a_0|, \ldots, |a_m|\}}$. 
We denote by $\bar\Q$ the field of complex algebraic numbers.  In particular,  unless the contrary is indicated, in this article ``algebraic number'' means ``complex algebraic number''.  The \textit{minimal polynomial} of ${\alpha \in \bar\Q}$ is, by definition,   the $\Z$-irreducible\footnote{A polynomial ${P(T)\in \Z[T]}$ is called \textit{$\Z$-irreducible} if it is irreducible in the ring $\Z[T]$.  The Gauss lemma implies that ${P(T):=a_mT^m+\cdots+a_0\in \Z[T]}$ is $\Z$-irreducible if and only if it is irreducible in $\Q[T]$ and ${\gcd(a_0, \ldots, a_m)=1}$.} polynomial vanishing at~$\alpha$; we denote it by  $\Mu_\alpha(T)$. It is well-defined up to multiplication by $\pm1$.

Now let~$m$ be a positive integer. We denote by $U_m$ the set of complex transcendental numbers~$\gamma$ with the following property:~$m$ is the smallest positive integer such that for every ${w>0}$ there exist infinitely many algebraic numbers~$\alpha$ of degree~$m$ satisfying ${|\gamma-\alpha|<|\Mu_\alpha|^{-w}}$. Informally,~$\gamma$ can be very well approximated by algebraic numbers of degree~$m$, but not by algebraic numbers of smaller degree. 
The numbers from the set $U_m$ are commonly called \textit{$U_m$-numbers}. In particular, the Liouville numbers  are exactly the $U_1$-numbers.

Furthermore, we set 
${U:=\bigcup_{m=1}^\infty U_m}$
and call the numbers from this set \textit{$U$-numbers}. Informally, $U$-numbers are those that can be very well approximated by algebraic numbers of bounded degree.


In this article we will consider  power series 
$$
f(z)=\sum_{k\geq 0}a_kz^k \in \Z[[z]]
$$ 
with integer coefficients and positive convergence radius ${R_f>0}$. We will always assume that $f(z)$ has infinitely many non-zero coefficients: it does not degenerate to a polynomial. 
We call the series $f(z)$ 
\textit{strongly lacunary}\footnote{Mahler~\cite{Ma65} calls such series \textit{admissible}.} if there exist integer sequences $(s_n)_{n\ge1}$ and $(t_n)_{n\ge0}$ such that 
\begin{align}
\label{emon}
&0\le  t_0\le s_1 <t_1\le s_2<t_2\le s_3\ldots,\\
\label{ezer}
&a_k=0\quad\text{for}\quad k<t_0 \quad\text{and for}\quad s_n<k<t_n \quad (n=1,2,\ldots), \\
\label{elacu}
&\frac{t_n}{s_n}\to \infty \quad  \text{as} \quad n\to \infty.
\end{align}
In particular, $f$ can be written as 
$$
f(z)=\sum_{n= 1}^\infty P_n(z), \quad \text{where}\quad P_n(z):=\sum_{k=t_{n-1}}^{s_n}a_kz^k.
$$
We will use the notation
$$
F_n(z):= \sum_{k=0}^{s_n} a_kz^k= P_1(z)+\cdots +P_n(z). 
$$
For instance, the classical Liouville series 
\begin{equation}
\label{eliouvilleser}
\sum_{n=1}^\infty z^{n!}
\end{equation}
is  a strongly lacunary series, with ${s_n:=n!}$ and ${t_n=(n+1)!}$. More generally, a series  ${\sum_{n=1}^\infty b_nz^{s_n}}$  with ${s_{n+1}/s_n\to \infty}$ is called \textit{diagonal} strongly lacunary series; for it 
${t_n=s_{n+1}}$ and  ${P_n(z)=b_nz^{s_n}}$.


Generalizing Liouville's  argument, Mahler~\cite{Ma65}
proved the following. 

\begin{theorem}
\label{thmahler}
Let $f(z)$ be a strongly  lacunary power series with integer coefficients  and let ${\alpha\in \bar\Q}$  satisfy ${|\alpha|<R_f}$. Assume that 
\begin{equation}
\label{econdmal}
\text{$P_n(\alpha) \ne 0$ for infinitely many~$n$}. 
\end{equation}
Then $f(\alpha)$ is transcendental. 
\end{theorem}

Clearly, condition~\eqref{econdmal} is necessary for $f(\alpha)$ being transcendental; Mahler proved that it is sufficient as well. 

In fact, Mahler's argument implies that $f(\alpha)$ is a $U$-number, because is is well approximated by infinitely many algebraic numbers $F_n(\alpha)$ of bounded degree. 

Mahler, however, does not specify to which of the sets $U_m$ belongs $f(\alpha)$.  If~$\alpha$ is of degree~$m$, then ${f(\alpha)\in U_{m'}}$ for some ${m'\le m}$, and simple examples show that~$m'$ can be strictly smaller than~$m$. Say, for the  series
${f(z):=\sum_{n=2}^\infty z^{n!}}$
we have ${f(1/\sqrt2)\in U_1}$. 

The objective of this article is to give a  sufficient condition for $f(\alpha)$ to be an $U_m$-number. We do it  for a slightly narrower, than  strongly lacunary, class of series. We call  $f(z)$ a \textit{generalized Liouville series}  if there exist integer sequences $(s_n)_{n\ge1}$ and $(t_n)_{n\ge0}$ satisfying~\eqref{emon},~\eqref{ezer},~\eqref{elacu} and the additional condition
\begin{equation}
\label{estrlac}
\text{the quotients}\quad \frac{s_{n+1}}{t_n}\quad(n\ge 1)\quad \text{are bounded}.
\end{equation}

\begin{remark}
To put it very informally, condition~\eqref{elacu} can be interpreted as ``the series $f(z)$ has very long blocks of consecutive zero coefficients'', while~\eqref{estrlac}   means that these blocks occur ``sufficiently often''
\end{remark}

Note that the diagonal strongly lacunary series, mentioned above, and, in particular, the Liouville series~\eqref{eliouvilleser}, belong to the class of generalized Liouville series: for them ${s_{n+1}/t_n=1}$.

Our principal result is the following theorem. 
\begin{theorem}
\label{thus}
Let $f(z)$ be a  generalized Liouville series with integer coefficients, and~$m$ the smallest integer with the following property:  among the numbers  $F_n(\alpha)$ there are infinitely many distinct algebraic numbers of degree~$m$. Then ${f(\alpha)\in U_m}$. 
\end{theorem}

Here are some consequences.  The first is a direct generalization of Liouville's result. 

\begin{corollary}
\label{coliou}
Let~$m$ be a positive integer,  ${\alpha \in \Q}$  satisfying ${0<\alpha<1}$, and ${\alpha^{1/m}\in \C}$ some determination of the $m^\tho$ root. Assume that~$\alpha$  is not a $p^\tho$ power in~$\Q$ for any prime ${p\mid m}$. Then ${f(\alpha^{1/m}) \in U_m}$, where $f(z)$ is the Liouville series~\eqref{eliouvilleser}. For instance, ${f(2^{-1/m})\in U_m}$. 
\end{corollary}

\begin{remark}
Corollary~\ref{coliou} is not really new; for instance, it can be deduced from (a very special case of) Theorem~1   of Alniaçik~\cite{Al79}. For more on this, see \cite{CM14,CMT21} and the references therein. 
\end{remark}

Call an algebraic number~$\alpha$ \textit{simple} if the number field $\Q(\alpha)$ does not have a proper subfield: ${\Q\subset K\subset\Q(\alpha)}$ implies ${K=\Q}$ or ${K=\Q(\alpha)}$.  Theorem~\ref{thus} implies that, for simple algebraic \textit{integers} of degree~$m$, Mahler's condition~\eqref{econdmal} is sufficient for ${f(\alpha)\in U_m}$.

\begin{corollary}
\label{cosim}
Let $f(z)$ be a  generalized Liouville series with integer coefficients and~$\alpha$ a simple algebraic integer of degree~$m$  satisfying ${|\alpha|<R_f}$. Assume that~\eqref{econdmal} holds. Then $f(\alpha)$ is a $U_m$-number.
\end{corollary}

Following Mahler~\cite{Ma65}, we say that the power series $f(z)$ is of \textit{bounded type}  if its coefficients are bounded: there exists ${A>0}$ such that 
${|a_k|\le A}$ for all~$k$. For such a series we have ${R_f=1}$.  

\begin{corollary}
\label{coalmost}
Let $f(z)$ be a  generalized Liouville series  of bounded type with integer coefficients, and~$m$ a positive integer. Then  ${f(\alpha) \in U_m}$ for all but finitely many simple algebraic integers~$\alpha$ of degree~$m$ with ${|\alpha|<1}$. 
\end{corollary}

Combining this with the results S.~D.~Cohen \cite{Co80,Co81,Co81a} on the distribution of Galois groups, we deduce 
that a generalized Liouville series of bounded type take $U_m$-values at ``many'' algebraic numbers of degree~$m$. Here is the precise statement.  

\begin{corollary}
\label{comany}
Let $f(z)$ be a  generalized Liouville series  of bounded type with integer coefficients and ${m\ge 2}$ an integer. Then,   for ${X\ge 2}$, there exist at least ${2^mX^{m-1}-CX^{m-3/2}\log X}$ algebraic numbers~$\alpha$ of degree~$m$ with ${|\alpha|<1}$ and  ${|\Mu_\alpha|\le X}$, such that $f(\alpha)$ is a $U_m$-number. Here~$C$ is a positive number depending only on~$f$ and~$m$. 
\end{corollary}

The bounded type hypothesis cannot be suppressed, because there exists a generalized Liouville series~$f$ with ${R_f=1}$ such that ${f(\alpha)\in \bar\Q}$ for every algebraic~$\alpha$ in the circle ${|z|<1}$.  Indeed, let ${U_1(T), U_2(T), \ldots \in \Z[T]}$ be a sequence of polynomials containing every non-constant polynomial with integer coefficients, and set 
${Q_n(T) := U_1(T)\cdots U_n(T)}$. Then  for every algebraic number~$\alpha$ we have ${Q_n(\alpha)=0}$ for all sufficiently big~$n$. Next,  let $(t_n)_{n\ge 0}$
be an increasing sequence of positive integers with the  properties
$$
t_{n-1} \ge \deg Q_{n}, \quad t_n > (n+1) t_{n-1}, \quad \max\{|Q_1|, \ldots |Q_{n}|\}^{1/t_{n-1}}\le 1+ \frac1{n} 
$$
for ${n\ge 1}$. We set  ${s_n := t_{n-1}+\deg Q_n}$, so that~\eqref{emon},~\eqref{ezer},~\eqref{elacu} are satisfied and ${s_n < 2t_{n-1}}$. 
Then the series 
$$
f(z):=\sum_{n=1}^\infty z^{t_{n-1}}Q_n(z) 
$$
is generalized Liouville  with ${R_f=1}$, and it takes algebraic values at algebraic numbers. See Sections~5 and~6 of Mahler~\cite{Ma65} for more elaborate examples. 

\paragraph{Plan of the article}
In Section~\ref{she} we recall the definitions and the basic facts about heights. In Section~\ref{sgap} we establish our principal technical tool, the \textit{gap principle}. 
Theorem~\ref{thus} is proved in Section~\ref{sproof}, and the corollaries are proved in Section~\ref{scols}.

\paragraph{Acknowledgments} 
We thank Joachim König for drawing our attention to the article~\cite{Co80}, and Yann Bugeaud for a stimulating discussion. We also thank the anonymous referee for careful reading and useful suggestions.

\section{Heights}
\label{she}
Mahler~\cite{Ma65} uses the notion of the ``height'' of an algebraic number~$\alpha$, defining it as $|\Mu_\alpha|$, the maximum of absolute values of the coefficients of the minimal polynomial. This notion of height was common at that time, but it is no longer in use, being replaced by the much more convenient notion of \textit{absolute logarithmic height}. In this section we recall the definition and basic of the absolute logarithmic height. In particular, we translate Mahler's definition of the $U_m$-numbers into the language of absolute logarithmic height. 

We start by fixing some conventions on absolute values.  Given a number field~$K$, we denote by $M_K$ the set of non-trivial absolute values on~$K$ normalized to extend the standard $p$-adic and infinite absolute values on~$\Q$. That is, if ${v\in M_K}$ and~$p$ is the underlying rational prime (finite or infinite), then 
${|p|_v=p^{-1}}$ if ${p<\infty}$ and\footnote{Here $1965$  is the year of publication of Mahler's article~\cite{Ma65}} ${|1965|_v=1965}$ if
${p=\infty}$.
With the chosen normalization, the product formula looks like 
\begin{equation}
\label{eprofo}
\prod |\alpha|_v^{d_v}=1 \qquad (\alpha\in K^\times). 
\end{equation}
Here $d_v$ denotes the\textit{ local degree} of~$v$ over~$\Q$: if ${v\mid p}$, then ${d_v:=[K_v:\Q_p]}$.

Let ${\alpha \in  \bar\Q}$ and let~$K$ be a number field containing~$\alpha$. Its \textit{absolute logarithmic height} (in the sequel, simply \textit{height}) $\height(\alpha)$ is defined by 
\begin{equation}
\label{ehedef}
\height(\alpha):= \frac{1}{[K:\Q]} \sum_{v\in M_K} d_v \log^+|\alpha|_v, 
\end{equation}
where ${\log^+ :=\max\{\log, 0\}}$. It is known that the right-hand side of~\eqref{ehedef} depends only on~$\alpha$, but not not on the particular choice of  the field~$K$.

For the basic properties of the height function the reader may consult, for instance, Sections~1.5 and~1.6 of~\cite{BG06}\footnote{Warning: in~\cite{BG06} a different normalization of absolute values is used.}. Below we list only the properties that will be used in this article. 
All the results of this section are well-known, but in some cases we provide quick proofs if we could not find a suitable reference.

\begin{proposition}
\label{prhe}
\begin{enumerate}
\item
\label{irat}
If ${m/n}$ is a rational number with  ${\gcd(m,n)=1}$, then ${\height(m/n)=\log \max\{|m|,|n|\}}$. In particular, if~$n$ is a non-zero integer then ${\height(n)=\log|n|}$. 

\item
\label{isumm}
Let ${\alpha, \beta \in \bar\Q}$. Then ${\height(\alpha+\beta) \le \height(\alpha)+\height(\beta)+\log 2}$.

\item
\label{ipower}
Let ${\alpha \in \bar\Q}$ and ${n\in \Z}$ (with ${\alpha \ne 0}$ if ${n<0}$). Then ${\height(\alpha^n)=|n|\height(\alpha)}$. In particular, for ${\alpha\ne 0}$ we have ${\height(\alpha^{-1})=\height(\alpha)}$.

\item
\label{igal}
Assume that~$\alpha$ and~$\beta$ are conjugate over~$\Q$. Then ${\height(\alpha)=\height(\beta)}$. 

\item
\label{inorth}
Let~$C$ be a positive number. Then there exist at most finitely many algebraic numbers~$\alpha$ with both height and degree not exceeding~$C$. 

\item
\label{ikro}
We have ${\height(\alpha) =0}$ if and only if ${\alpha=0}$ or~$\alpha$ is a root of unity. 
\end{enumerate}
\end{proposition}

Item~\ref{inorth} is usually called \textit{Northcott's theorem}, and item~\ref{ikro} is \textit{Kronecker's theorem}. 

\begin{proof}
Item~\ref{irat} follows easily from the definition.  The remaining items  are in~\cite{BG06}, see Proposition~{1.5.15}, Lemma~{1.5.18}, Proposition~{1.5.17}, Theorem~{1.6.8} and  Theorem~{1.5.9} therein. 
\end{proof}

\begin{proposition}[Liouville's inequality]
\label{prli}
Let~$\beta$ and~$\beta'$ be distinct complex algebraic numbers of degrees~$d$ and~$d'$, respectively. Then 
$$
|\beta-\beta'| \ge e^{-dd'(\height(\beta)+\height(\beta')+\log2)}. 
$$
\end{proposition}

In particular, in the case ${d'=1}$ we recover the classical Liouville's equality proved in~\cite{Li44,Li51}. 

The proof can be found, for instance, in \cite[Theorem~{1.5.21}]{BG06}.

We will also be using the notion of the \textit{height of a polynomial}. Given 
$$
P(T)=a_nT^n+\cdots +a_0\in \bar\Q[T],
$$
we define
\begin{equation}
\label{edefhop}
\height(P):= \frac{1}{[K:\Q]} \sum_{v\in M_K} d_v \log^+|P|_v, 
\end{equation}
where~$K$ is a number field containing the coefficients ${a_0, \ldots, a_m}$ and 
$$
|P|_v: =\max\{|a_0|_v, \ldots, |a_n|_v\}. 
$$
As before, the right-hand side of~\eqref{edefhop} does not depend on the particular choice of the field~$K$.

The height defined in~\eqref{edefhop} is sometimes called the \textit{affine height}. If~$P$ is a \textit{non-zero} polynomial, then we can also define its \textit{projective height}: 
$$
 \heightp(P):= \frac{1}{[K:\Q]} \sum_{v\in M_K} d_v \log|P|_v. 
$$
It follows from the definition that  ${\heightp(P) \le \height(P)}$. By the product formula, we have ${\heightp(\lambda P)=\heightp(P)}$ for ${\lambda\in \bar\Q^\times}$.  For a monic polynomial ${T-\alpha}$ of degree~$1$ we have 
\begin{equation}
\label{epodegone}
\height(T-\alpha)=\heightp(T-\alpha) =\height(\alpha). 
\end{equation}

\begin{remark}
\label{reintpol}
For a non-zero polynomial 
${P(T)=a_nT^n+\cdots +a_0\in \Z[T]}$
with integer coefficients we have 
\begin{equation}
\label{ehepoz}
\height(P) = \log |P|, \qquad \heightp(P)= \log |P|-\log \delta(P), 
\end{equation}
where 
${|P|:= \max\{|a_0|, \ldots, |a_n|\}}$ and ${\delta(P):= \gcd(a_0, \ldots, a_n)}$. 
\end{remark}

\begin{proposition}
\label{prhpol}
Let ${P(T) \in \bar\Q[T]}$ be a non-zero polynomial of degree~$n$. 

\begin{enumerate}
\item
\label{ihpal}
For ${\alpha \in \bar\Q}$ we have 
\begin{equation}
\label{ehpal}
\height\bigl(P(\alpha)\bigr) \le n\height(\alpha) +\height(P)+ \log(n+1). 
\end{equation}

\item
\label{ipalzer}
Let ${\alpha \in \bar\Q}$ be a root of~$P$. Then 
${\height(\alpha)\le \heightp(P)+\log 2}$. 

\item
\label{iprodpol}
Let ${Q_1(T), \ldots, Q_s(T) \in \bar\Q[T]}$ be such that ${P=Q_1\cdots Q_s}$. Then 
$$
\bigl|\heightp(P)-\bigl(\heightp(Q_1)+\cdots+\heightp(Q_s)\bigr)\bigr|\le n\log 2. 
$$

\item
\label{iroots}
Let ${\alpha_1, \ldots, \alpha_n\in \bar\Q}$ be the roots of $P(T)$ counted with multiplicities. Then 
\begin{equation}
\label{ehepheas}
\bigl|\heightp(P)-\bigl(\height(\alpha_1)+\cdots+\height(\alpha_n)\bigr)\bigr|\le n\log 2. 
\end{equation}

\end{enumerate}
\end{proposition}

\begin{proof}
To prove item~\ref{ihpal}, let~$K$ be a number field containing~$\alpha$ and the coefficients of~$P$. Then for ${v\in M_K}$ we have 
$$
\max\{1, |P(\alpha)|_v\} \le \max\{1, |n+1|_v\}\max\{1,|P|_v\}\max\{1, |\alpha|_v\}^n. 
$$
Taking the logarithm and summing up, we obtain~\eqref{ehpal}. 

For item~\ref{ipalzer} see \cite[Proposition~3.6]{BB13}\footnote{Proposition~3.6 in~\cite{BB13} is based on Proposition~3.2 therein. We use this opportunity to indicate a minor typo in the proof of the latter: in the displayed equation on the bottom of page~162, ${1-|\alpha|}$ in the denominator must be ${1-|\alpha|^{-1}}$.}. 

Item~\ref{iprodpol} is (a special case of) Theorem~1.6.13 from~\cite{BG06}.  To prove item~\eqref{iroots},  we may assume that~$P$ is monic, because multiplying~$P$ by  ${\lambda\in \bar\Q^\times}$  does not affect $\height_p(P)$. Hence
${P(T)=(T-\alpha_1)\cdots (T-\alpha_n)}$, 
and~\eqref{ehepheas} follows from~\eqref{epodegone} and item~\ref{iprodpol}. 
\end{proof}

\begin{corollary}
\label{cmihe}
Let~$\alpha$ be an algebraic number of degree~$n$ and ${\Mu_\alpha(T)\in \Z[T]}$ its minimal polynomial, as defined in Section~\ref{sintro}. Then 
$
{\bigl|\log|\Mu_\alpha|-n\height(\alpha)\bigr| \le n\log 2} 
$.
\end{corollary}

\begin{proof}
This follows from Remark~\ref{reintpol}, item~\ref{igal} of Proposition~\ref{prhe} and item~\ref{iroots} of Proposition~\ref{prhpol}. 
\end{proof}

\begin{remark}
\label{rebridge}
Corollary~\ref{cmihe} provides a bridge between the old-fashioned height used in~\cite{Ma65} and the absolute logarithmic height. Using it, we can give an equivalent definition of $U_m$-numbers, that will be used in  Sections~\ref{sgap} and~\ref{sproof}. 
A complex transcendental number~$\gamma$ is called $U_m$-number if~$m$ is the smallest integer with the following property: for every ${w>0}$ there exist infinitely many  algebraic~$\beta$ of degree~$m$ such that ${|\gamma-\beta|\le e^{-w\height(\beta)}}$. %
\end{remark}

\section{A gap principle}
\label{sgap}

Our principal technical tool is the following \textit{gap principle}.

\begin{theorem}
\label{thgap}
Let~$\gamma$ be a complex transcendental number and~$d$ a positive number. 
Let $(\beta_n)_{n\ge1}$ be a sequence of algebraic numbers of degree not exceeding~$d$ converging to~$\gamma$ (in the complex topology). Let $(w_n)_{n\ge1}$ be a sequence  of positive real numbers with ${w_n\to \infty}$, such that 
\begin{equation}
\label{eapprogabe}
|\gamma-\beta_n|\le Ce^{-w_n\height(\beta_n)} \qquad (n\ge1)
\end{equation}
with some ${C>0}$. 
Assume that there exists ${A>0}$ such that 
\begin{equation}
\label{egap}
\height(\beta_{n+1})\le Aw_n \height(\beta_n) \qquad (n \ge 1). 
\end{equation}  
Then there exists positive numbers~$c$ and~$\eta$ such that for every algebraic~$\beta$  of degree not exceeding~$d$, distinct from  each of~$\beta_n$, we have
\begin{equation}
\label{elowerbebe}
|\gamma-\beta|\ge ce^{-\eta\height(\beta)}.
\end{equation}
\end{theorem}

Very informally, if~$\gamma$ has  a ``sufficiently dense'' set of very good approximations by algebraic numbers of bounded degree, then it cannot be well approximated by the algebraic numbers of bounded degree not belonging to this set.

We do not know if this statement can be found in the literature verbatim, but arguments of this flavor has been known since long ago;  see, for instance,  \mbox{LeVeque's}\ proof~\cite{Le53} of existence of $U_m$-numbers; a simpler exposition can be found   in Bugeaud~\cite[Theorem~7.4]{Bu04}. Another example is the work of Alniaçik~\cite{Al79} about the values of rational functions at Liouville numbers. 

To illustrate the idea of the proof in its purity, 
let us consider a very special case
\begin{equation}
\label{efeliou}
d=1, \qquad \gamma:=\sum_{n=1}^\infty 2^{-n!}, \qquad \beta_n:=\sum_{k=1}^n 2^{-k!}. 
\end{equation}

\begin{proposition}
\label{prantibu}
Let $p/q$ be a rational number, distinct from each of $\beta_n$, as defined in~\eqref{efeliou}. Then 
\begin{equation}
\label{elowerqq}
\left|\gamma-\frac pq\right|\ge \frac c{q^3}, 
\end{equation}
where ${c>0}$ is an absolute constant. 
\end{proposition}

Thus, in this special case Theorem~\ref{thgap} holds with ${\eta=3}$. 

\begin{proof}
Write the rational number $\beta_n$ as $p_n/q_n$, where ${q_n=2^{n!}}$.
A simple estimate shows that  ${|\gamma-p_n/q_n|\le 2q_n^{-n-1}}$. 
Now let~$n$ be such that 
$$
q_n\le q^2<q_{n+1}=q_n^{n+1}.
$$ 
Since ${p/q\ne p_n/q_n}$ by the hypothesis, we have ${|p/q-p_n/q_n|\ge 1/qq_n}$.  It follows that 
$$
\left|\gamma-\frac pq\right|\ge \left|\frac pq -\frac{p_n}{q_n}\right|- \left|\gamma-\frac{p_n}{q_n}\right| \ge \frac{1}{qq_n}-\frac{2}{q_n^{n+1}}. 
$$
When ${n\ge n_0}$, where~$n_0$ is an absolute constant, we have 
$$
\frac{2}{q_n^{n+1}} \le \frac{1/2}{q_n^{(n+3)/2}} \le \frac{1/2}{qq_n}. 
$$
Hence for ${q\ge q_0:= (q_{n_0})^{1/2}}$ we have 
$$
\left|\gamma-\frac pq\right|\ge \frac{1/2}{qq_n}\ge  \frac{1/2}{q^3}. 
$$
Thus, for ${q\ge q_0}$, inequality~\eqref{elowerqq} holds with ${c=1/2}$. After adjusting~$c$, it holds  for   ${q<q_0}$ as well. 
\end{proof}

\begin{remark}
\label{rexpl}
The key point of this argument is properly choosing~$n$. We define~$n$ from inequality ${q_n\le q^2<q_{n+1}}$.  In fact, the same argument works if we use 
\begin{equation}
\label{elamm}
q_n\le q^\lambda<q_{n+1}
\end{equation} 
with any ${\lambda>1}$. In this case, exponent~$3$ in~\eqref{elowerqq} becomes ${1+\lambda}$. Thus, one can obtain a lower bound of the shape ${|\gamma-p/q|\ge c(\eps)q^{-2-\eps}}$ for any ${\eps>0}$. 
\end{remark}

\begin{proof}[Proof of Theorem~\ref{thgap}]
To start with, let us show that 
\begin{equation}
\label{eheinf}
\height(\beta_n) \to \infty
\end{equation}
as ${n\to \infty}$.  Indeed, assume the contrary: there exists ${B>0}$ such that ${\height(\beta_n) \le B}$ for infinitely many~$n$. Northcott's theorem (item~\ref{inorth} of Proposition~\ref{prhe}) implies that there exist at most finitely many algebraic numbers of degree bounded by~$d$ and height bounded by~$B$. It follows that infinitely many terms of the sequence $(\beta_n)$ are equal to the same algebraic number. Since ${\beta_n \to \gamma}$, this implies that~$\gamma$ is an algebraic number, a contradiction. This proves~\eqref{eheinf}.

Next, since ${\beta\ne \beta_n}$ by the hypothesis, Proposition~\ref{prli} implies that for all~$n$ we have 
$$
|\beta-\beta_n| \ge c_1e^{-d^2(\height(\beta)+\height(\beta_n))}, \qquad c_1:= 2^{-d^2}. 
$$
Hence 
\begin{equation}
\label{egamibe}
|\gamma-\beta|\ge |\beta-\beta_n|-|\gamma-\beta_n| \ge c_1e^{-d^2(\height(\beta)+\height(\beta_n))} - Ce^{-w_n\height(\beta_n)}. 
\end{equation}
Since both ${\height(\beta_n)\to \infty}$ and ${w_n\to \infty}$, there exists a positive integer~$n_0$ such that for all ${n\ge n_0}$ we have 
\begin{equation}
\label{ebig}
\frac12c_1e^{-d^2\height(\beta_n)} >C e^{-(w_n/2)\height(\beta_n)}. 
\end{equation}

Set ${\lambda:= 2d^2A}$, where~$A$ is from~\eqref{egap}. Assume first that 
\begin{equation}
\label{ebebig}
\lambda \height(\beta) \ge \height(\beta_{n_0})
\end{equation}
In this case the set of indices ${n\ge n_0}$ with the property ${\lambda\height(\beta) \ge \height(\beta_{n})}$ is not empty. Since ${\height(\beta_n)\to \infty}$, there exists ${n\ge n_0}$ with the property 
\begin{equation}
\label{ellla}
\height(\beta_n) \le \lambda \height(\beta)  < \height(\beta_{n+1}). 
\end{equation}
We fix one such~$n$. (One may compare~\eqref{ellla} with~\eqref{elamm}.) 

Our definition of~$\lambda$ together with~\eqref{egap} implies that 
$$
d^2\height(\beta) < \frac{1}{2A}\height(\beta_{n+1}) \le \frac12 w_n\height(\beta_n). 
$$
Combining this with~\eqref{ebig}, we obtain
$$
\frac12c_1e^{-d^2(\height(\beta)+\height(\beta_n))} \ge Ce^{-w_n\height(\beta_n)}. 
$$
Together with~\eqref{egamibe}, this implies that
$$
|\gamma-\beta|\ge \frac12{c_1}e^{-d^2(\height(\beta)+\height(\beta_n))}  \ge \frac12{c_1} e^{-d^2(1+\lambda)\height(\beta)}. 
$$
Thus, in the case~\eqref{ebebig} inequality~\eqref{elowerbebe} holds with ${\eta:=d^2(1+\lambda)}$ and ${c:=c_1/2}$.  

To complete the proof, note that~\eqref{ebebig} may fail for at most finitely many~$\beta$  by Northcott.  By adjusting~$c$, the lower bound~\eqref{elowerbebe} extends to these~$\beta$ as well. 
\end{proof}

\section{Proof of Theorem~\ref{thus}}
\label{sproof}
In this section, unless the contrary is stated explicitly, we adopt the set-up of Theorem~\ref{thus}. Thus, ${f(z)=\sum_{k=0}^\infty a_kz^k}$ is a power series with integer coefficients having convergence radius ${R_f>0}$, and there exist integer sequences $(s_n)_{n\ge 1}$, $(t_n)_{n\ge 0}$ satisfying~\eqref{emon},~\eqref{ezer},~\eqref{elacu}  and~\eqref{estrlac}.  We use notation $P_n(z)$ and $F_n(z)$ as defined in Section~\ref{sintro}. 
We let~$\alpha$ be an algebraic number with ${|\alpha|<R_f}$, satisfying~\eqref{econdmal}. 
In what follows, constants implied by 
the Vinogradov notation ``$\ll$'' may depend on~$f$,~$\alpha$ and~$d$.

We will be using Theorem~\ref{thgap} with 
$$
d:=[\Q(\alpha):\Q], \qquad \gamma:=f(\alpha), \qquad \beta_n:=F_n(\alpha).
$$ 
Note that~$\gamma$ is a transcendental number by Theorem~\ref{thmahler}. 

To start with, we   estimate  the coefficients $a_k$ and the differences ${\gamma-\beta_n}$. Set 
\begin{equation*}
r:=\frac{R_f+|\alpha|}{2}, \qquad \theta := \log \frac{r}{|\alpha|}, \qquad M:=\max_{|z|=r}|f(z)|. 
\end{equation*}
From 
$$
a_k=\frac1{2\pi i}\int_{|z|=r}\frac{f(z)}{z^{k+1}}dz
$$
we bound 
\begin{equation}
\label{eboundak}
|a_k|\le Mr^{-k}, 
\end{equation}
which implies that ${|a_k\alpha^k|\le Me^{-\theta k}}$. Now 
\begin{equation}
\label{ediffer}
|\gamma-\beta_n| =\left|\sum_{k=t_n}^\infty a_k\alpha^k\right|\le M e^{-\theta t_n}\sum_{k=0}^\infty e^{-\theta k} = Ce^{-\theta t_n}
\end{equation}
with 
${C:= M/(1-e^{-\theta})}$.

Next, let us estimate the height and the degree of the polynomial $F_n(z)$:
$$
\deg F_n \le s_n, \qquad \height(F_n) \ll s_n, 
$$
where the estimate for degree follow from the definition of~$F_n$, and the estimate for the height follows from~\eqref{eboundak} by Remark~\ref{reintpol}.  Using item~\ref{ihpal} of Proposition~\ref{prhpol}, we deduce the estimate
\begin{equation}
\label{ehefnab}
\height(\beta_n)=\height\bigl(F_n(\alpha) \bigr) \ll s_n. 
\end{equation}

To apply Theorem~\ref{thgap}, we
need to verify that the numbers~$w_n$ can be defined to satisfy~\eqref{eapprogabe} and~\eqref{egap}.  
Arguing as in the proof of Theorem~\ref{thgap}, we show that ${\height(\beta_n) \to \infty}$ as ${n\to \infty}$. In particular, ${\height(\beta_n) >0}$ for all but finitely many~$n$. 
We set ${w_n:= \theta t_n/\height(\beta_n)}$ if ${\height(\beta_n) >0}$. For the finitely many~$n$ with ${\height(\beta_n) =0}$ we set ${w_n:=1}$. With this definitions~\eqref{ediffer} implies that
\begin{equation}
\label{elf}
|\gamma-\beta_n| \le Ce^{-w_n\height(\beta_n)} \qquad (n\ge 1), 
\end{equation}
which is~\eqref{eapprogabe}. Condition~\eqref{elacu} together with~\eqref{ehefnab} implies that ${w_n\to \infty}$, while condition~\eqref{estrlac} together with~\eqref{ehefnab} implies that ${\height(\beta_{n+1})\ll w_n\height(\beta_n)}$, which is~\eqref{egap}.
Thus, all the hypotheses of Theorem~\ref{thgap} are satisfied.

Now we are ready to complete the proof of Theorem~\ref{thus}.  As indicated in Remark~\ref{rebridge}, we have to prove that for every ${w>0}$  inequality 
\begin{equation}
\label{egoodapp}
|\gamma-\beta|<e^{-w\height(\beta)}
\end{equation}
holds for infinitely many algebraic~$\beta$ of degree~$m$, and for some ${w>0}$ inequality~\eqref{egoodapp} holds for at most finitely many algebraic~$\beta$ of degree strictly smaller than~$m$.

Recall that, by the assumption, among the numbers~$\beta_n$ infinitely many are of degree~$m$, but at most finitely many are of smaller degree. 

Let ${w>0}$. Since ${w_n\to \infty}$ and ${\height(\beta_n)\to \infty}$, we have ${|\gamma-\beta_n|\le C e^{-w\height(\beta_n)}}$ for all but finitely many~$n$. Since, among the numbers~$\beta_n$,  infinitely many are of degree~$m$,  inequality~\eqref{egoodapp} holds for infinitely many algebraic~$\beta$ of degree~$m$.

Next,   there exist ${\eta>0}$ and ${c>0}$ such that inequality
${|\gamma-\beta|\ge ce^{-\eta\height(\beta)}}$
holds for all algebraic~$\beta$ of degree smaller than~$m$. Indeed, Theorem~\ref{thgap} implies that it is true for all such~$\beta$ that are distinct from any of the $\beta_n$. Since at most finitely many of  the numbers $\beta_n$  are of degree  smaller than~$m$, it  holds for these finitely many numbers as well, after adjusting~$c$.  It follows that, when ${w>\eta}$,  inequality~\eqref{egoodapp} holds for at most finitely many algebraic~$\beta$ of degree strictly smaller than~$m$. 
Theorem~\ref{thus} is proved.

\section{Proof of corollaries}
\label{scols}

The proof of Corollary~\ref{coliou} requires a  lemma. Given a field~$K$, an element ${\lambda \in K}$ and a positive integer~$m$, we use the notation 
${\lambda K^m:=\{\lambda\beta^m:\beta \in K\}}$. 

\begin{lemma}
\label{lradic}
Let~$K$ be a field of characteristic~$0$ and  ${\alpha \in K}$. Let~$m$ be a positive integer. Fix some determination of the $m^\tho$ root $\alpha^{1/m}$ (in an algebraic closure of~$K$), and for every ${d\mid m}$ define ${\alpha^{1/d}:= (\alpha^{1/m})^{m/d}}$.
\begin{enumerate}

\item
\label{ilang}
 Assume 
that for every prime  ${p\mid m}$ we have ${\alpha\notin K^p}$. If ${4\mid m}$ then we assume, in addition, that ${\alpha\notin -4K^4}$. Then  $\alpha^{1/m}$ is of degree~$m$ over~$K$. 

\item
\label{iconrad}
Assume that $\alpha^{1/m}$ is of degree~$m$ over~$K$,  and  that every root of unity from the field $K(\alpha^{1/m})$ is contained in~$K$. Then every field in between~$K$ and $K(\alpha^{1/m})$ is of the form $K(\alpha^{1/d})$ for some ${d\mid m}$.

\end{enumerate}

\end{lemma}

The hypothesis about the roots of unity in item~\ref{iconrad} cannot be suppressed. For instance, let ${K=\Q}$, ${m=4}$ and  ${\alpha =-1}$, in which case ${\alpha^{1/m}=\zeta_8}$, the primitive $8^\tho$ root of unity. Then the fields $\Q(\sqrt2)$ and $\Q(\sqrt{-2})$ are contained in $\Q(\zeta_8)$, but they are not of the form $\Q(\zeta_8^d)$ for some ${d\mid 4}$.

\begin{proof}
Item~\ref{ilang} is Theorem~9.1  in~\cite[Chapter~VI]{La02}. The proof of item~\ref{iconrad} can be found, for instance, in Conrad's online note~\cite{Co25}. For the reader's convenience, we prove it here as well.
Since $\alpha^{1/m}$ is of degree~$m$ over~$K$, the degree of $\alpha^{1/d}$ over~$K$ must be~$d$ for every ${d\mid m}$. Now let~$L$ be an intermediate field: ${K\subset L\subset K(\alpha^{1/m})}$. Set ${d:=[L:K]}$, so that ${[K(\alpha^{1/m}):L]=m/d}$. The conjugates of $\alpha^{1/m}$ over~$K$ are all of the form $\alpha^{1/m}\zeta$, where~$\zeta$ is an $m^\tho$ root of unity. Hence 
$$
\NN_{K(\alpha^{1/m})/L}(\alpha^{1/m}) = (\alpha^{1/m})^{m/d}\xi= \alpha^{1/d}\xi,
$$ 
where~$\xi$ is a root of unity. 

We have ${\xi \in K(\alpha^{1/m})}$; by the hypothesis, this implies that ${\xi \in K}$. Hence ${\alpha^{1/d}\in L}$. Comparing the degrees, we obtain ${L=K(\alpha^{1/d})}$. 
\end{proof}

\begin{proof}[Proof of Corollary~\ref{coliou}]
It suffices to prove that 
$$
\beta:=F_{m-1}(\alpha^{1/m}) = \sum_{k=1}^{m-1}(\alpha ^{1/m})^{k!}
$$
is algebraic of degree~$m$, because for ${n\ge m}$ we have ${F_n(\alpha^{1/m}) -\beta \in \Q}$. 

By the hypothesis, ${\alpha \notin \Q^p}$ for every prime ${p\mid m}$; also, ${\alpha \notin -4\Q^2}$  because ${\alpha >0}$. Item~\ref{ilang} of Lemma~\ref{lradic} implies that $\alpha^{1/m}$ is of degree~$m$, and all the $m^\tho$ roots of~$\alpha$ are conjugate over~$\Q$. Since ${\alpha>0}$, at least one of these roots is real, which implies that the field $\Q(\alpha^{1/m})$ cannot have roots of unity other than $\pm1$. Item~\ref{iconrad} of Lemma~\ref{lradic} implies that the only subfields of $\Q(\alpha^{1/m})$ are $\Q(\alpha^{1/d})$ for ${d\mid m}$.

Since ${\beta \in \Q(\alpha^{1/m})}$, we can write 
$$
\beta = a_0+a_1\alpha^{1/m} + \cdots + a_{m-1}(\alpha^{1/m})^{m-1}, \qquad a_0, \ldots, a_{m-1} \in \Q. 
$$
Moreover, we have the explicit formula
$$
a_\ell = \sum_{\genfrac{}{}{0pt}{}{1\le k\le m-1}{k!\equiv \ell \bmod m}}\alpha^{(k!-\ell)/m} \qquad (\ell= 0, 1, \ldots, m-1). 
$$
In particular, for ${\ell=1}$ the sum on the right  contains the term ${\alpha^{(1!-1)/m}=1}$. Since ${\alpha>0}$, this implies that ${a_1\ge 1}$.

Now if~$\beta$ is of degree~$d$, then we must have ${d\mid m}$ and ${\beta \in \Q(\alpha^{1/d})}$. This implies that ${a_\ell=0}$ unless~$\ell$ is divisible by $m/d$. Since ${a_1\ne 0}$, this is only possible when ${d=m}$. 
\end{proof}

\begin{proof}[Proof of Corollary~\ref{cosim}]
Let~$\alpha$ be a simple algebraic integer of degree~$m$, satisfying~\eqref{econdmal}.  To start with, let us note that the set ${\{F_n(\alpha): n\ge 1\}}$ is infinite. Indeed, if it is  finite, then the sequence $F_n(\alpha)$ is a convergent sequence with finitely many distinct terms, which means that it has equal terms from some point on. Hence ${P_n(\alpha)=0}$ for sufficiently large~$n$,  which contradict~\eqref{econdmal}.

Now, by Theorem~\ref{thus}, it suffices to prove  that for all but finitely many~$n$,  the algebraic numbers 
$F_n(\alpha)$ are  of degree~$m$. 
Assume the contrary: for infinitely many~$n$ the number $F_n(\alpha)$ is of degree smaller than~$m$. Since~$\alpha$ is simple, for those~$n$ we must have ${F_n(\alpha)\in \Q}$. Since~$\alpha$ is an algebraic integer, so is $F_n(\alpha)$. Thus, ${F_n(\alpha) \in \Z}$ for infinitely many~$n$.  
Since ${F_n(\alpha)\to f(\alpha)}$, this implies that ${f(\alpha) \in \Z}$. Since~$\alpha$ satisfies~\eqref{econdmal}, this contradicts Theorem~\ref{thmahler}.
\end{proof}

Corollary~\ref{coalmost} follows immediately from Corollary~\ref{cosim} and the following simple observation, which might be of independent interest. 

\begin{proposition}
\label{prantima}
Let~$f$ be strongly lacunary series of bounded type, and~$d$ a positive number.  Then~\eqref{econdmal} holds (and $f(\alpha)$ is thereby transcendental) for all but finitely many algebraic~$\alpha$ of degree not exceeding~$d$. 
\end{proposition}

\begin{proof}
Since~$f$ is of bounded type, there exists ${A>0}$ such that ${|a_k|\le A}$ for all~$k$. Hence ${\heightp(P_n) \le \log A}$ for all~$n$, see Remark~\ref{reintpol}. Item~\ref{ipalzer} of Proposition~\ref{prhpol} now implies that ${\height(\alpha) \le \log(2A)}$ for every~$\alpha$ not satisfying~\eqref{econdmal}. By Northcott's theorem (item~\ref{inorth} of Proposition~\ref{prhe}), there can be at most finitely many such~$\alpha$ of degree bounded by~$d$. 
\end{proof}

\begin{remark}
Bounding the degree in Proposition~\ref{prantima} is crucial: a strongly lacunary series of bounded type may take algebraic values at infinitely many algebraic numbers, see Section~6 of Mahler~\cite{Ma65}. 
\end{remark}

Corollary~\ref{comany} follows from Corollary~\ref{coalmost} and a  result on distribution of Galois groups, essentially due to S.~D.~Cohen. By the\textit{ Galois group} of a  polynomial ${P(T)\in \Q[T]}$ of degree~$m$ we mean  its  Galois group over~$\Q$, viewed as a  subgroup of the symmetric group $S_m$. 

\begin{proposition}
\label{pcohen}
Let ${m\ge 2}$ and let~$b$ be a nonzero integer. Let  ${X\ge 2}$. Then, among the  monic polynomials ${P(T)\in \Z[T]}$ of degree~$m$, satisfying 
$$
|P|\le X, \qquad P(0)=b, 
$$
at most  ${O(X^{m-3/2}\log X)}$  
may have Galois group other than~$S_m$. The implicit constant here depends  on~$m$ and~$b$. 
\end{proposition}

\begin{proof}

Assume first that ${m=2}$. A polynomial of degree~$2$ has Galois group~$S_2$ if and only if it is irreducible. There exist at most finitely many monic reducible polynomials of degree~$2$ with ${P(0)=b}$, because the roots of such a polynomial are divisors of~$b$. Hence, for ${m=2}$ the statement holds even with $O(1)$ estimate. 

From now on ${m\ge 3}$.  We consider the polynomial 
$$
P(T,u_1,\ldots, u_{m-1}) := T^m+u_{m-1}T^{m-1}+\cdots +u_1T_1+b, 
$$
with indeterminate coefficients $u_1, \ldots, u_{m-1}$. Its Galois group over the field $\Q(u_1, \ldots, u_{m-1})$ is~$S_n$, see Theorem~1 in Cohen~\cite{Co80}\footnote{There is a minor mistake in~\cite{Co80}, corrected in~\cite{Co81}, but Theorem~1, that we are quoting, is not affected.} (take therein~$u_1$ and~$u_2$ as~$u$ and~$v$, and $\Q(u_3, \ldots, u_{m-1})$ as~$F$). By another result of Cohen, Theorem~2.1 from~\cite{Co81a},  there are at most  $O(X^{m-3/2}\log X)$ vectors ${(b_1, \ldots, b_{m-1})\in \Z^{m-1}}$, satisfying 
${\max\{|b_1|, \ldots, |b_{m-1}|\}\le X}$,  
such that the Galois group of the polynomial  ${T^m+b_{m-1}T^{m-1}+\cdots +b_1T+b}$ is not~$S_m$.  This completes the proof. 
\end{proof}

The recent breakthrough work of Bhargava and others (see~\cite{Bh25} and the references therein) counts polynomials with non-symmetric Galois groups among all the monic polynomials~$P$ of degree~$m$, obtaining estimates which are quantitatively sharper than those of Cohen (and, essentially, best possible). However, we need to impose the extra condition ${P(0)=b}$, which makes the results of Bhargava not suitable for our purposes.

\begin{proof}[Proof of Corollary~\ref{comany}]

There exist at least ${2(2X-1)^{m-1}}$ monic  polynomials ${P(T)\in \Z[T]}$ satisfying 
\begin{equation}
\label{epolls}
|P|\le X, \qquad P(0)\in \{1,-1\}
\end{equation}
Proposition~\ref{pcohen} implies that at least 
$$
2(2X-1)^{m-1} -C_1X^{m-3/2}\log X\ge  2^mX^{m-1} -C_2X^{m-3/2}\log X
$$
of them have Galois group~$S_m$.  Here  $C_1$ and $C_2$ are positive numbers  that  depend on~$m$.

Let $P(T)$ be one of these polynomials, that is, it is monic, satisfies~\eqref{epolls} and has Galois group~$S_m$; in particular, $P(T)$ is $\Z$-irreducible.  We may assume that $P(T)$ is not a cyclotomic polynomial. Indeed, for ${m\ge 3}$ it cannot be cyclotomic, because its Galois group $S_m$ is not abelian. There do exist~$3$ cyclotomic polynomials  with Galois group~$S_2$, but we can exclude them from consideration by adjusting the constant~$C_2$. 

Since $P(T)$ is monic, not cyclotomic and satisfies ${|P(0)|=1}$, Kronecker's theorem (item~\ref{ikro} of Proposition~\ref{prhe}) implies that $P(T)$
must have a root~$\alpha$ satisfying ${|\alpha|<1}$. Since~$P$ is irreducible,  ${P(T)=\Mu_\alpha(T)}$ is the minimal polynomial of~$\alpha$; in particular ${|\Mu_\alpha|=|P|\le X}$. Since~$P$ is monic,~$\alpha$ is an algebraic integer. Since the Galois group of $P(T)$ is symmetric,~$\alpha$ is simple. (This is because the subgroup of $S_m$ stabilizing~$\alpha$ is~$S_{m-1}$, which is known to be a maximal subgroup of~$S_m$.) 

Thus, we have ${ 2^mX^{m-1} -C_2X^{m-3/2}\log X}$ simple algebraic integers~$\alpha$ of degree~$m$ satisfying ${|\alpha|<1}$ and ${|\Mu_\alpha|\le X}$.  Corollary~\ref{coalmost} implies that ${f(\alpha)\in U_m}$  for all these~$\alpha$ with finitely many exceptions. This completes the proof.  
\end{proof}

{\footnotesize

\bibliographystyle{amsplain}
\bibliography{gen_liou}

\bigskip

\noindent 
Yuri Bilu: 
Institut de Mathématiques de Bordeaux, Université de Bordeaux \& CNRS, Talence, France; 
\url{yuri@math.u-bordeaux.fr}

\bigskip

\noindent
Diego Marques: Departamento de Matemática, Universidade de Brasília, Brasília, DF, Brazil; 
\url{diego@mat.unb.br}

\bigskip

\noindent
Carlos Gustavo Moreira: Instituto de Matemática Pura e Aplicada, Rio de Janeiro, RJ, Brazil; 
\url{gugu@impa.br}

}

\end{document}